\documentclass[a4paper,twoside,11pt]{amsart}
\pdfoutput=1 
\usepackage[english]{babel}
\usepackage[latin1]{inputenc}
\usepackage{amsmath}
\usepackage{amssymb}
\usepackage{amsfonts}
\usepackage{mathrsfs}
\usepackage{hyperref}

\usepackage{enumerate}

\let\itemref\ref

\usepackage[matrix,arrow,cmtip,2cell]{xy} 
\SelectTips{cm}{}
\newdir{(}{{}*!/-5pt/@^{(}}
\newdir{(x}{{}*!/-5pt/@_{(}}
\newdir{+}{{}*!/-9pt/{}}
\newdir{>+}{@{>}*!/-9pt/{}}
\entrymodifiers={+!!<0pt,\fontdimen22\textfont2>}
\UseAllTwocells

\usepackage[nosubsections]{mytheorems2}

\newtheoremstyle{dtheoremnopar}{3 mm}{1 mm}{\itshape}{}{\bfseries}{.}{ }
{\thmname{#1}\thmnumber{ #2}\thmnote{ \mdseries(#3)\bfseries}}

\theoremstyle{dtheoremnopar}
\newcounter{theoremx}
\newtheorem{theoremalpha}[theoremx]{Theorem}
\newtheorem{corollaryalpha}[theoremx]{Corollary}

\newcommand{\tref}[1]{\ref{#1}} 

\newcommand\NN{\mathbb{N}}
\newcommand\ZZ{\mathbb{Z}}
\newcommand\inj{\hookrightarrow}
\newcommand\id[1]{\mathrm{id}_{#1}} 
\DeclareMathOperator{\Hom}{Hom}
\DeclareMathOperator{\Aut}{Aut}
\DeclareMathOperator{\Sym}{Sym}
\newcommand\sA{\mathcal{A}}
\newcommand\sB{\mathcal{B}}
\newcommand\sF{\mathcal{F}}
\newcommand\sG{\mathcal{G}}
\newcommand\sO{\mathcal{O}}
\newcommand\im{\mathfrak{m}}
\DeclareMathOperator{\Spec}{Spec}
\newcommand{\GL}{\mathrm{GL}}
\newcommand{\gitq}{/\!\!/} 
\newcommand\Hilb{\mathrm{Hilb}}
\newcommand\Quot{\mathrm{Quot}}
\newcommand\QCoh{\mathbf{QCoh}} 
\newcommand\Coh{\mathbf{Coh}} 
\newcommand\HilbSt{\mathscr{H}} 
\newcommand{\weilr}{\mathbf{R}} 
\newcommand\qfin{\mathrm{qfin}} 
\newcommand\catC{\mathbf{C}} 
\newcommand\Pos{\mathbf{Pos}} 
\newcommand\op{\mathrm{op}} 
\newcommand\APP{\mathbf{App}} 
\newcommand\Stk{\mathbf{Stk}} 
\newcommand\Sub{\mathbf{Sub}} 
\newcommand\fl{\mathrm{flat}} 
\newcommand\alg{\mathrm{alg}} 
\let\lim\relax
\let\im\relax
\DeclareMathOperator{\lim}{lim}
\DeclareMathOperator{\colim}{colim}
\DeclareMathOperator{\im}{im}
\DeclareMathOperator{\Fun}{Fun}
\DeclareMathOperator{\Map}{Map}
\DeclareMathOperator{\Isom}{Isom} 
\DeclareMathOperator{\ev}{ev} 

\newcommand\stX{\mathscr{X}} 
\newcommand\stY{\mathscr{Y}} 
\newcommand\stW{\mathscr{W}} 
\newcommand\stU{\mathscr{U}} 
\newcommand\stG{\mathscr{G}} 

\newcommand\Dqc{\mathbf{D}_{\mathrm{qc}}} 

\begin{document}

\title{Absolute noetherian approximation of algebraic stacks}
\author{David Rydh}
\address{KTH Royal Institute of Technology\\Department of 
  Mathematics\\SE-100 44 Stockholm\\Sweden}
\email{dary@math.kth.se}
\thanks{The author was supported by the Swedish Research Council 2015-05554 and the G\"oran Gustafsson Foundation for Research in Natural Sciences and Medicine.}
\date{2023-11-15}

\begin{abstract}
We show that every quasi-compact and quasi-separated algebraic stack
can be approximated by a noetherian algebraic stack. We give several
applications such as eliminating noetherian hypotheses in the theory
of good moduli spaces.
\end{abstract}

\maketitle


\setcounter{secnumdepth}{0}
\begin{section}{Introduction}
The purpose of this paper is to prove the following ultimate version of
absolute noetherian approximation for algebraic stacks,
confirming~\cite[Conj.~B]{rydh_approximation-sheaves}.
\begin{theoremalpha}\label{T:MT}
Let $X$ be an algebraic stack. The following are equivalent
\begin{enumerate}
\item\label{TI:qcqs}
  $X$ is quasi-compact and quasi-separated.
\item\label{TI:approx}
  There exists an algebraic stack $X_0$ of finite presentation
over $\Spec \ZZ$ and an affine morphism $X\to X_0$.
\end{enumerate}
\end{theoremalpha}
When~\itemref{TI:approx} holds, we say that $X$ can be approximated
and $X_0$ is an approximation of $X$~\cite[Def.~7.1]{rydh_noetherian-approx}.
Theorem~\tref{T:MT} was known for schemes, algebraic spaces,
Deligne--Mumford stacks and algebraic stacks with finite stabilizers,
see~\cite{rydh_noetherian-approx} and the references therein.

From basic results on stacks with approximation, we obtain
two corollaries. Firstly, we settle~\cite[Conj.~A]{rydh_approximation-sheaves}.
\begin{corollaryalpha}[Completeness property]\label{C:COMPLETE}
Let $X$ be a quasi-compact and quasi-separated algebraic stack. Then every
quasi-coherent $\sO_X$-module is a directed colimit of
$\sO_X$-modules of finite presentation. In particular, every quasi-coherent $\sO_X$-module of finite
type is a quotient of an $\sO_X$-module of finite presentation.
\end{corollaryalpha}
\begin{proof}
Let $h\colon X\to X_0$ be an approximation. Since $X_0$ is noetherian, every
quasi-coherent $\sO_X$-module is the union of its coherent
subsheaves~\cite[Prop.~15.4]{laumon}. The result then follows from considering
the counit of the adjunction
$(h^*,h_*)$, see~\cite[Prop.~4.6]{rydh_noetherian-approx}.
\end{proof}

Since the $\sO_X$-modules of finite presentation are exactly the compact
objects, Corollary~\tref{C:COMPLETE} says that $\QCoh(\sO_X)$ is compactly
generated, or in the terminology of~\cite{rydh_noetherian-approx}, that $X$ is
pseudo-noetherian.  Secondly, we have the following stronger version
of Theorem~\tref{T:MT}.
\begin{corollaryalpha}\label{C:MT-limit}
Let $f\colon X\to Y$ be a morphism between quasi-compact and quasi-separated
algebraic stacks. Then $f$ has an approximation, that is:
\begin{enumerate}
\item there exists a finitely presented morphism $X_0\to Y$ and an affine $Y$-morphism $X\to X_0$; and
\item there exists an inverse system $X_\lambda\to Y$ of finitely presented
  morphism with affine transition maps and inverse limit $X\to Y$.
\end{enumerate}
Moreover, if $f$ is of finite type, then in (i), $X\to X_0$ can be taken to be
a closed immersion and in (ii), the system can be chosen such that the
transition maps are closed immersions.
\end{corollaryalpha}
\begin{proof}
By the main theorem, $X$ and $Y$ can be approximated and are thus of strict
approximation type~\cite[Def.~2.9]{rydh_noetherian-approx}.
It follows that $f$ also is of strict approximation
type~\cite[Prop.~2.10 (vii), (viii)]{rydh_noetherian-approx} so
we can factor $f$ as $X\to X_0\to Y$~\cite[Prop.~4.8]{rydh_noetherian-approx}.
The remainder is now straightforward, and is spelled out
in~\cite[Prop.~7.3, 7.4]{rydh_noetherian-approx}.
\end{proof}

Corollary~\tref{C:MT-limit} also implies that 
\cite[Thm.~D]{rydh_noetherian-approx} applies to any morphism
between quasi-compact and quasi-separated algebraic stacks.
Another easy consequence is that every quasi-separated, but not necessarily
quasi-compact, morphism of algebraic stacks is
\emph{locally of approximation type}~\cite[\S1]{hall-rydh_gen-hilb-quot}.

\subsection*{Outline of the proof of the main theorem}
Let $X$ be a quasi-compact and quasi-separated algebraic stack. Let $U\to X$ be
a smooth presentation with $U=\Spec B$ an affine scheme. Then we can write $B$
as the union of its subalgebras $B_\lambda$ of finite type over $\ZZ$ and hence
$U=\varprojlim U_\lambda$ where $U_\lambda = \Spec(B_\lambda)$.  Unfortunately,
the inverse system $U_\lambda$ does not ``descend'' to an inverse system
$X_\lambda$ with limit $X$, not even for $\lambda$ large enough.

We solve this problem as follows. Fix a directed set $\Lambda$.  In
Section~\ref{S:almost-shape} we introduce colimits and limits of \emph{almost
  shape $\Lambda$}. These are diagrams $\{X_\lambda\}_{\lambda\geq \alpha}$ in
some category for some $\alpha\in \Lambda$. In Section~\ref{S:approximations}
we introduce the $2$-category $\APP_\Lambda(X)$ of approximations of $X$, that
is, diagrams $\{X_\lambda\}_{\lambda\geq \alpha}$ of almost shape $\Lambda$
with $X_\lambda\to \Spec \ZZ$ of finite presentation, affine transition maps
and limit $X$. We show that this $2$-category is equivalent to a
\emph{partially ordered set}. In Section~\ref{S:descent} we show that
$\APP_\Lambda$ satisfies faithfully flat descent: given a faithfully flat
morphism $X'\to X$, we have an equalizer:
\[
\xymatrix{%
\APP_\Lambda(X)\ar[r]
 & \APP_\Lambda(X')\ar@<.5ex>@{+->+}[r] \ar@<-.5ex>@{+->+}[r]
 & \APP_\Lambda(X'\times_X X').
}%
\]
The problem alluded to above is that if we start with an arbitrary
approximation of $X'$, it does not descend to $X$.

In Section~\ref{S:pure-adjoint}, we show that $f^*\colon \APP_\Lambda(Y)\to
\APP_\Lambda(X)$ admits a \emph{right} adjoint $f_*$ if $f\colon X\to Y$ is
smooth with \emph{geometrically connected fibers}. The existence of this adjoint
can be checked smooth-locally on $Y$ via descent. In this way, we can assume
that $Y$ already has an approximation $Y\to Y_0$ and in this case one can
write down an explicit formula for the adjoint.

The main theorem now follows from taking an arbitrary smooth presentation $U\to
X$ and considering its canonical factorization $U\to \pi_0(U/X)\to X$ (``smooth
d\'evissage''). The first morphism is smooth with connected fibers and handled
by the adjoint described above. The second morphism is \'etale and was dealt
with in~\cite{rydh_noetherian-approx} using \'etale d\'evissage.

The proof of the main theorem in~\cite{rydh_approximation-sheaves} followed a
similar smooth d\'evissage strategy. For a smooth morphism $f\colon X\to Y$
with geometrically connected fibers and a quasi-coherent sheaf $\sF$ on $Y$,
the functor $f^*\colon \Sub(\sF)\to \Sub(f^*\sF)$, on quasi-coherent
subsheaves, admits a left adjoint $f_!$. This left adjoint, extended to
subalgebras, is also used in Section~\ref{S:pure-adjoint}.

\subsection{Further applications}
In Section~\ref{S:applications} we use the main theorem to eliminate noetherian
or finite presentation assumptions in the following results:
\begin{enumerate}
\item Algebraicity of Quot schemes and Hom-stacks~\cite{hall-rydh_gen-hilb-quot,hall-rydh_coherent-tannaka-duality}.
\item Zariski's main theorem for stacks.
\item Proper coverings of separated stacks~\cite{olsson_proper-coverings}.
\item Local structure and other results on good moduli
  spaces~\cite{alper-hall-rydh_etale-local-stacks}.
\end{enumerate}
The main theorem has also already been used
in~\cite[Thm.~5.1]{alper-hall-halpern-leistner-rydh_local-third} to obtain a
general local structure theorem for algebraic stacks at points with linearly
reductive stabilizers without finiteness hypotheses.

\end{section}
\setcounter{secnumdepth}{3}


\begin{section}{Categories of limit diagrams}\label{S:almost-shape}
Let $\Lambda$ be a directed set and let $\catC$ be a $2$-category. In this
section we introduce the $2$-category $\colim'(\Lambda,\catC)$ of colimit
diagrams of $\catC$ of
``almost shape $\Lambda$''. Roughly speaking, an object of
$\colim'(\Lambda,\catC)$ is a colimit diagram of
shape $\Lambda_{\geq \alpha}$ for some $\alpha$ and two objects are equal if
the colimit diagrams agree after increasing $\alpha$.

We let $\Lambda^\triangleright=\Lambda \cup \{\infty\}$ where $\infty$ is
strictly larger then all elements of $\Lambda$. We say that
$\overline{p}\colon \Lambda^\triangleright\to \catC$ is a colimit diagram of
shape $\Lambda$ if it exhibits $\overline{p}(\infty)$ as a colimit of the
diagram $p=\overline{p}|_\Lambda$. Similarly, we consider limit diagrams of
shape $\Lambda^\op$ and note that $\infty\in ({\Lambda^\triangleright})^\op$ is
now the initial element.

We let $\colim(\Lambda,\catC)\subset \Fun(\Lambda^\triangleright,\catC)$ denote
the full $2$-subcategory of colimit diagrams of shape $\Lambda$ and similarly
for $\lim(\Lambda^\op,\catC)\subset \Fun((\Lambda^\triangleright)^\op,\catC)$.
For any
indices $\alpha\leq \beta$ in $\Lambda$, we have a restriction functor
$\colim(\Lambda_{\geq \alpha},\catC)\to \colim(\Lambda_{\geq \beta},\catC)$
and evaluating in $\infty$ gives us functors $\ev_\infty\colon \colim(\Lambda_{\geq \alpha},\catC)\to \catC$.
We consider the following $2$-categories of (co)limit diagrams of
``almost shape $\Lambda$''
\begin{align*}
\colim'(\Lambda,\catC) &= \colim_{\alpha\in \Lambda} \colim(\Lambda_{\geq \alpha},\catC)\\
\lim'(\Lambda^\op,\catC) &= \colim_{\alpha\in \Lambda} \lim( (\Lambda_{\geq \alpha})^\op,\catC)
\end{align*}
that also come with evaluation functors $\ev_\infty$ to $\catC$.
Let us explicitly spell out the $2$-category $\colim'(\Lambda,\catC)$.

A \emph{strict} object of $\colim'(\Lambda,\catC)$
consists of an index $\alpha\in \Lambda$, objects $A_\lambda\in
\catC$ for every $\lambda\geq \alpha$ and $1$-morphisms
$\varphi_{\lambda_1\lambda_2}\colon A_{\lambda_1}\to A_{\lambda_2}$ for every
$\lambda_2\geq \lambda_1\geq \alpha$ such that
$\varphi_{\lambda\lambda}=\id{A_\lambda}$ for every $\lambda\geq \alpha$ and
$\varphi_{\lambda_2\lambda_3}\varphi_{\lambda_1\lambda_2}=\varphi_{\lambda_1\lambda_3}$ for every $\lambda_3\geq \lambda_2\geq \lambda_1\geq \alpha$
and $A_\infty=\colim_{\alpha\leq \lambda<\infty}
A_\lambda$. We denote such an object by $\{A_\lambda\}_{\lambda \geq \alpha}$
suppressing the $\varphi$'s.

A \emph{non-strict} object is similar but the two conditions on $\varphi$ are
replaced with specified $2$-isomorphisms
$\eta_\lambda\colon \id{A_\lambda} \Rightarrow \varphi_{\lambda\lambda}$
and $\mu_{\lambda_1\lambda_2\lambda_3}\colon
\varphi_{\lambda_2\lambda_3}\varphi_{\lambda_1\lambda_2}\Rightarrow
\varphi_{\lambda_1\lambda_3}$
such that
$\mu_{\lambda_1\lambda_2\lambda_2}(\eta_{\lambda_2}\star \id{\varphi_{\lambda_1\lambda_2}})=\id{\varphi_{\lambda_1\lambda_2}}
=\mu_{\lambda_1\lambda_1\lambda_2}(\id{\varphi_{\lambda_1\lambda_2}}\star \eta_{\lambda_1})$ and
$\mu_{\lambda_1\lambda_3\lambda_4}(\id{\varphi_{\lambda_3\lambda_4}}\star\mu_{\lambda_1\lambda_2\lambda_3})
=\mu_{\lambda_1\lambda_2\lambda_4}(\mu_{\lambda_2\lambda_3\lambda_4}\star\id{\varphi_{\lambda_1\lambda_2}})$
for every $\lambda_4\geq \lambda_3\geq \lambda_2\geq \lambda_1 \geq \alpha$.

A $1$-morphism $\{A_\lambda\}_{\lambda \geq \alpha}\to \{B_\lambda\}_{\lambda
  \geq \beta}$ consists of an index $\gamma\in \Lambda$, greater than $\alpha$
and $\beta$, together with $1$-morphisms $f_\lambda\colon A_\lambda\to
B_\lambda$ for every $\lambda \geq \gamma$ and $2$-isomorphisms
$\tau_{\lambda_1\lambda_2}\colon
\varphi_{\lambda_1\lambda_2}f_{\lambda_1}\Rightarrow
f_{\lambda_2}\varphi_{\lambda_1\lambda_2}$ such that
$(\id{f_{\lambda_3}}\star \mu_{\lambda_1\lambda_2\lambda_3})(\tau_{\lambda_2\lambda_3}\star \tau_{\lambda_1\lambda_2})=\tau_{\lambda_1\lambda_3}(\mu_{\lambda_1\lambda_2\lambda_3}\star \id{f_{\lambda_1}})$
for every $\lambda_3\geq \lambda_2\geq \lambda_1 \geq \gamma$. We denote such an
object by $\{f_\lambda\}_{\lambda \geq \gamma}$, suppressing the $\tau$'s.

A $2$-morphism $\{f_\lambda\}_{\lambda\geq \gamma}\Rightarrow
\{g_\lambda\}_{\lambda\geq \delta}$ is an equivalence class of the set
consisting of an index $\epsilon \in \Lambda$ greater than $\gamma$ and
$\delta$, together with $2$-morphisms $\rho_\lambda\colon f_\lambda \Rightarrow
g_\lambda$ for every $\lambda \geq \epsilon$ such that
$\tau_{\lambda_1\lambda_2}(\id{\varphi_{\lambda_1\lambda_2}}\star \rho_{\lambda_1})=
(\rho_{\lambda_2}\star \varphi_{\lambda_1\lambda_2})\tau_{\lambda_1\lambda_2}$ for every $\epsilon\leq \lambda_1\leq
\lambda_2$. We denote such an element by $\{\rho_{\lambda}\}_{\lambda\geq
  \epsilon}$ and two elements are equivalent if they agree after increasing
$\epsilon$.

Note that $\alpha,\beta,\gamma,\delta,\epsilon\in \Lambda$ whereas
$\lambda,\lambda_1,\lambda_2,\lambda_3\in \Lambda^{\triangleright}$.

When $\catC$ is the $2$-category of categories, or the $2$-category of
categories fibered over a fixed category, then the usual coherence results
show that $\colim'(\Lambda,\catC)$ is equivalent to the full $2$-subcategory
of strict objects.
\end{section}


\begin{section}{The category of approximations}\label{S:approximations}
Let $\Lambda$ be a directed set and let $\Stk$ denote the $2$-category of
algebraic stacks. In this section, we study the category of limit diagrams of
algebraic stacks of almost shape $\Lambda$ with the following additional
assumptions:

\begin{definition}
The \emph{category of approximations}
$\APP_\Lambda$ of almost shape $\Lambda$ is the full $2$-subcategory of
$\lim'(\Lambda^\op,\Stk)$ consisting of objects $\{X_\lambda\}_{\lambda\geq
  \alpha}$ such that
\begin{enumerate}
\item $X_\lambda$ is of finite presentation over $\Spec \ZZ$ for all
  $\lambda\in \Lambda_{\geq \alpha}$, and
\item $\varphi_{\lambda\mu}\colon X_\mu\to X_\lambda$
  is affine and schematically dominant for all $\mu\geq \lambda\geq \alpha$.
\end{enumerate}
\end{definition}
Note that $\APP_\Lambda$ is a $(2,1)$-category since $\Stk$ is a
$(2,1)$-category. The affine morphism $\varphi_{\lambda\mu}\colon X_\mu\to
X_\lambda$ corresponds to a homomorphism $A_\lambda\to A_\mu$ of
$\sO_{X_\alpha}$-algebras and that $\varphi_{\lambda\mu}$ is schematically
dominant means that $A_\lambda\to A_\mu$ is injective. All the $X_\lambda$,
including $X_\infty$, are quasi-compact and quasi-separated so it makes no
difference if we consider limits in $\Stk$ or its full $2$-subcategory of
quasi-compact and quasi-separated stacks.

\begin{example}\label{E:approximation}
Let $X$ be an algebraic stack that can be approximated, that is, there
exists an algebraic stack $X_0$ of finite presentation over $\Spec \ZZ$ and an
affine morphism $h\colon X\to X_0$. Replacing $X_0$ by its schematic image, we
may assume that $h$ is schematically dominant. Let
$A_\infty=h_*\sO_X$. Let $\Lambda$ be the directed set of
$\sO_{X_0}$-subalgebras $A_\lambda\subseteq A_\infty$ of finite type. Then
$A_\infty=\varinjlim_\lambda A_\lambda$ since $X_0$ is noetherian
\cite[Prop.~15.4]{laumon}. If we let $X_\lambda=\Spec_{X_0}(A_\lambda)$, then
$\{X_\lambda\}$ is an element of $\APP_\Lambda$ with $X_\infty=X$.
\end{example}

The $2$-category $\APP_\Lambda$ is equivalent to the following $2$-category
where the intermediate $X_\lambda$, $\alpha<\lambda<\infty$, are replaced with
subalgebras.

An object $\{X_\lambda\}_{\lambda\geq \alpha}$ is an algebraic stack $X_\alpha$
of finite presentation over $\Spec \ZZ$ together with an affine schematically
dominant morphism $\varphi_\alpha\colon X_\infty\to X_\alpha$ and a colimit
diagram of finitely generated subalgebras $A_\lambda\subseteq A_\infty=(\varphi_{\alpha})_*
\sO_{X_\infty}$ of shape $\Lambda_{\geq \alpha}$. Here
$A_\lambda=(\varphi_{\alpha\lambda})_* \sO_{X_\lambda}$ and we do not need to
specify any $2$-isomorphisms analogous to
$\eta_\lambda$ and $\mu_{\lambda_1\lambda_2\lambda_3}$.

A $1$-morphism $\{f_\lambda\}_{\lambda\geq \gamma} \colon
\{X_\lambda\}_{\lambda\geq \alpha}\to \{Y_\lambda\}_{\lambda\geq \beta}$ is a
$2$-commutative diagram
\[
\xymatrix{%
  X_\infty\ar[r]^{f_\infty}\ar[d]_{\varphi_\gamma}\drtwocell\omit{\quad\tau_{\gamma\infty}} & Y_\infty\ar[d]^{\varphi_\gamma}\\
  X_\gamma\ar[r]_{f_\gamma} & Y_\gamma
}%
\]
such that if $B_\lambda\subseteq B_\infty$ and $A_\lambda\subseteq A_\infty$
are the algebras over $\sO_{X_\gamma}$ and $\sO_{Y_\gamma}$ corresponding to
$X_\lambda$ and $Y_\lambda$, then the image of
$f_\gamma^*A_\lambda\to f_\gamma^*A_\infty\to B_\infty$
is contained in $B_\lambda\subseteq B_\infty$ for every $\lambda\geq \gamma$.

A $2$-morphism $\{\rho_\lambda\}_{\lambda\geq \epsilon}\colon
\{f_\lambda\}_{\lambda\geq \gamma}\Rightarrow \{g_\lambda\}_{\lambda\geq
  \delta}$ are $2$-morphisms $\rho_\infty\colon f_\infty\Rightarrow g_\infty$
and $\rho_\epsilon\colon f_\epsilon\Rightarrow g_\epsilon$ such that
$\tau_{\epsilon\infty}(\id{\varphi_\epsilon}\star \rho_\infty)=
(\rho_\epsilon\star \id{\varphi_\epsilon})\tau_{\epsilon\infty}$. Two
$2$-morphisms are identified if they agree after increasing $\epsilon$.

As before we have an evaluation map $\ev_\infty\colon \APP_\Lambda\to \Stk$.

\begin{proposition}\label{P:fully-faithful-on-morphisms}
Let $\{X_\lambda\}$ and $\{Y_\lambda\}$ be objects of $\APP_\lambda$.  Then the
map of groupoids $\Map_{\APP_\lambda}(\{X_\lambda\},\{Y_\lambda\})\to
\Map_\Stk(X_\infty,Y_\infty)$ is fully faithful.
\end{proposition}
\begin{proof}
Let $\{f_\lambda\}_{\lambda}$ and $\{g_\lambda\}_{\lambda}$ be two morphisms
$\{X_\lambda\}\to \{Y_\lambda\}$. For both objects and morphisms, we may
assume that $\lambda\geq \alpha$ for some fixed $\alpha$.
We need to show that $\Phi\colon \Map(\{f_\lambda\}_{\lambda},\{g_\lambda\}_{\lambda})\to
\Map(f_\infty,g_\infty)$ is bijective.

Let $I_\lambda=\Isom_{X_\lambda}( f_\lambda, g_\lambda
)=X_\lambda\times_{Y_\lambda\times Y_\lambda,\Delta} Y_\lambda$. The natural
map $I_\lambda\to X_\lambda$ is representable and of finite presentation. We
identify $2$-morphisms $\rho_\lambda\colon f_\lambda\Rightarrow g_\lambda$ with
sections of $I_\lambda\to X_\lambda$. Let $\mu\geq\lambda\geq \alpha$.  Since
$\varphi_{\lambda\mu}\colon X_\mu\to X_\lambda$ is affine, the induced morphism
$I_\mu\to I_\lambda\times_{X_\lambda} X_\mu$ is a closed immersion.

We first show that $\Phi$ is injective.
Suppose we are given $\{(\rho_i)_\lambda\}\colon \{f_\lambda\}\Rightarrow
\{g_\lambda\}$ for $i=1,2$ with $\lambda\geq \alpha$, such that
$(\rho_1)_\infty=(\rho_2)_\infty$. Identifying $(\rho_i)_\lambda$ with
sections of $I_\lambda\to X_\lambda$, we have
$(\rho_1)_\alpha\varphi_{\alpha\infty}=\varphi_{\alpha\infty}(\rho_1)_\infty
=\varphi_{\alpha\infty}(\rho_2)_\infty=(\rho_2)_\alpha\varphi_{\alpha\infty}$.
The two sections $(\rho_i)_\alpha\varphi_{\alpha\infty}$ of
$I_\alpha\times_{X_\alpha} X_\infty\to X_\infty$ thus coincide.
Since $I_\alpha\to X_\alpha$ is of
finite presentation, it follows that the two sections
$(\rho_i)_\alpha\varphi_{\alpha\lambda}$ of $I_\alpha\times_{X_\alpha}
X_\lambda\to X_\lambda$ coincide for all sufficiently large $\lambda$.
Since $I_\lambda\to I_\alpha\times_{X_\alpha} X_\lambda$ is a monomorphism,
it follows that $(\rho_1)_\lambda=(\rho_2)_\lambda$ for all sufficiently
large $\lambda$.

We now show that $\Phi$ is surjective.
Let $\rho_\infty\colon f_\infty\Rightarrow g_\infty$ be a $2$-morphism and
identify it with a section of $I_\infty\to X_\infty$. We have a
map $\varphi_{\alpha\infty}\rho_\infty\colon X_\infty\to I_\infty \to I_\alpha$
and since $I_\alpha\to X_\alpha$ is of finite presentation, it factors through
a map $\widetilde{\rho}_\epsilon\colon X_\epsilon \to I_\alpha$ for some
$\epsilon\geq \alpha$. Let $\lambda \geq \epsilon$. Then we have an induced
section $\rho_\lambda\colon X_\lambda\to I_\alpha\times_{X_\alpha} X_\lambda$.
But $X_\infty\to X_\lambda$ is schematically dominant, $I_\lambda\to
I_\alpha\times_{X_\alpha} X_\lambda$ is a closed immersion and
$\varphi_{\alpha\infty}\rho_\infty$ factors through $I_\lambda$. It follows
that $\rho_\lambda$ factors uniquely through $I_\lambda$. Thus, we have
a $2$-morphism $\{\rho_\lambda\}_{\lambda \geq \epsilon}$.
\end{proof}

Proposition~\ref{P:fully-faithful-on-morphisms} says that given
$\{X_\lambda\}$ and $\{Y_\lambda\}$ and $f_\infty\colon X_\infty\to Y_\infty$
there exists at most one morphism $\{f_\lambda\}$ above $f_\infty$ up to
unique $2$-isomorphism. In particular we have:

\begin{corollary}
Let $X$ be an algebraic stack. The $2$-category
$\APP_\Lambda(X):=\ev_\infty^{-1}(X)$ is a preordered set,
hence equivalent to a partially ordered set.
\end{corollary}

If $\{X_\lambda\},\{Y_\lambda\}\in \APP_\Lambda(X)$, then we write
$\{X_\lambda\}\geq \{Y_\lambda\}$ if there exists a morphism
$\{X_\lambda\}\to \{Y_\lambda\}$.

\begin{remark}\label{R:approximation-vs-APP-nonempty}
If $\APP_\Lambda(X)$ is not empty, then we have an affine morphism $X\to
X_\alpha$ with $X_\alpha$ of finite presentation over $\ZZ$, that is, $X$ has
an approximation. Conversely, if $X$ has an approximation,
then we can find a directed set $\Lambda$ such that $\APP_\Lambda(X)$ is
non-empty, cf.\ Example~\ref{E:approximation}.
\end{remark}

\begin{corollary}\label{C:approximations=colimits-of-subalgebras}
Let $X$ be an algebraic stack. Let $h\colon X\to X_0$ be an affine
schematically dominant morphism to an algebraic stack $X_0$ of finite
presentation over $\ZZ$.
Let $A=h_*\sO_X$. Then $\APP_\Lambda(X)^\op$ is equivalent to the category
$\APP_\Lambda(A)$ of colimit diagrams
$\{A_\lambda\}$ of finitely generated $\sO_{X_0}$-subalgebras of $A$ of almost
shape $\Lambda$ and colimit $A_\infty=A$.
\end{corollary}
\begin{proof}
We have already seen that $\APP_\Lambda(X)$ is a preordered set and a
colimit diagram $\{A_\lambda\}$ gives rise to an element of $\APP_\Lambda(X)$ by
applying $\Spec_{X_0}(-)$. Conversely, given $\{X_\lambda\}_{\lambda\geq \alpha}$ we may
factor $X_\infty=X\to X_0$ through $X_\gamma$ for some $\gamma\geq
\alpha$. After increasing $\gamma$, we can arrange so that $X_\gamma\to X_0$ is
affine~\cite[Thm.~C]{rydh_noetherian-approx}. Let $A_\lambda$ be the
push-forward of the structure sheaf along $X_\lambda\to X_\gamma\to X_0$ for
every $\lambda\geq \gamma$. Then $\{A_\lambda\}_{\lambda\geq \gamma}$ is an
object of $\APP_\Lambda(A)$.
\end{proof}

If $\{A_\lambda\},\{B_\lambda\}\in \APP_\Lambda(A)$, then we write
$\{A_\lambda\}\leq \{B_\lambda\}$ if there exists a morphism $\{A_\lambda\}\to
\{B_\lambda\}$, or equivalently, if $A_\lambda\subseteq B_\lambda$ for all
sufficiently large $\lambda$. Note that the convention is opposite to that in
$\APP_\Lambda(X)$.

\begin{remark}
Two elements $\{A_\lambda\}_{\lambda\geq \alpha}$ and
$\{B_\lambda\}_{\lambda\geq \beta}$ have least upper bound $\{A_\lambda\cup
B_\lambda\}_{\lambda \geq \gamma}$ where $\gamma$ is some upper bound of
$\alpha$ and $\beta$. In particular, if $\APP_\Lambda(X)$ is non-empty, then
$\APP_\Lambda(X)$ is an upper semi-lattice. However, $\APP_\Lambda(X)$ is not
necessarily a lattice. Even if $X_0=\Spec k$ is the spectrum of a field and $A$
is a $k$-algebra, the intersection of two finitely generated sub-algebras of
$A$ need not be finitely generated.
\end{remark}

If $X$ is of finite presentation over $\Spec \ZZ$, then $\APP_\Lambda(X)$ is the
singleton set for any directed set $\Lambda$. If $X$ is not quasi-compact
and quasi-separated, then $\APP_\Lambda(X)$ is empty.
\end{section}


\begin{section}{The stack of approximations}\label{S:descent}
We have seen that the fibers of $\ev_\infty\colon \APP_\Lambda\to \Stk$ are
preordered sets. In this section, we show that $\APP_\Lambda\to \Stk$ is
a stack after restricting to flat and finitely presented morphisms.

\subsection{Fibered category}
We begin by studying cartesian arrows in $\APP_\Lambda$.
\begin{definition}
We say that a $1$-morphism $\{f_\lambda\}\colon \{X_\lambda\}\to \{Y_\lambda\}$
in $\APP_\Lambda$
is \emph{cartesian} if there exists an index $\alpha\in \Lambda$ such that
for every $\mu\geq \lambda\geq \alpha$, the square
\[
\xymatrix{%
X_\mu\ar[r]^{f_\mu}\ar[d]^{\varphi_{\lambda\mu}}
  & Y_\mu\ar[d]^{\varphi_{\lambda\mu}}\\
X_\lambda\ar[r]^{f_\lambda} & Y_\lambda
}%
\]
is cartesian.
\end{definition}

The cartesian $1$-morphisms are cartesian in the sense of fibered categories:

\begin{proposition}\label{P:cartesian}
Let $\{h_\lambda\}\colon \{X_\lambda\}\to \{Z_\lambda\}$ and
$\{g_\lambda\}\colon \{Y_\lambda\}\to \{Z_\lambda\}$ be $1$-morphisms in
$\APP_\Lambda$ and suppose that we are given a commutative diagram
\[
\xymatrix@R=5mm{
\rrtwocell\omit{<3.5>\rho} & Y_\infty\ar[dr]^{g_\infty} & \\
X_\infty\ar[ur]^f\ar[rr]_{h_\infty} && Z_\infty.
}%
\]
If $\{g_\lambda\}$ is cartesian, then there exists a $1$-morphism
$\{f_\lambda\}\colon \{X_\lambda\}\to \{Y_\lambda\}$ and a $2$-isomorphism
$\{\rho_\lambda\}\colon \{g_\lambda f_\lambda\} \Rightarrow \{h_\lambda\}$ such
that $f_\infty=f$ and $\rho_\infty=\rho$. Moreover, $(f_\lambda,\rho_\lambda)$
is unique up to unique $2$-isomorphism.
\end{proposition}
\begin{proof}
The uniqueness is Proposition~\ref{P:fully-faithful-on-morphisms}.  For the
existence, pick an index $\alpha$ for which $\{Y_\lambda\}$ is defined.  The
composition $X_\infty\to Y_\infty\to Y_\alpha$ then factors through
$X_\gamma\to Y_\alpha$ for some $\gamma$ that we can take to be larger than
$\alpha$. After increasing $\gamma$, the two compositions $X_\gamma\to
Y_\alpha\to Z_\alpha$ and $X_\gamma\to Z_\gamma\to Z_\alpha$ coincide up to
some $2$-isomorphism compatible with $\rho$. Since $\{g_\lambda\}$ is
cartesian, this gives, for every $\lambda\geq \gamma$, a map $f_\lambda\colon
X_\lambda\to Y_\lambda=Y_\alpha\times_{Z_\alpha} Z_\lambda$ together with a
$2$-isomorphism $\rho_\lambda\colon g_\lambda f_\lambda \Rightarrow h_\lambda$.
\end{proof}

We also have lots of cartesian arrows:

\begin{proposition}
Let $\{Y_\lambda\}$ be an object in $\APP_\Lambda$.  If $f\colon X\to Y_\infty$
is a flat morphism of finite presentation, then there exists a cartesian arrow
$\{f_\lambda\}\colon \{X_\lambda\}\to \{Y_\lambda\}$ such that $X_\infty=X$ and
$f_\infty=f$ and $(\{X_\lambda\}, \{f_\lambda\})$ is unique up to unique
$2$-isomorphism.
\end{proposition}
\begin{proof}
The uniqueness is Proposition~\ref{P:fully-faithful-on-morphisms}.  For the
existence, we note that $f\colon X_\infty\to Y_\infty$ descends to a flat and
finitely presented morphism $f_\alpha\colon X_\alpha \to
Y_\alpha$ for some $\alpha$~\cite[App.~B]{rydh_noetherian-approx}. For every
$\lambda\in \Lambda_{\geq \alpha}$, we let
$X_\lambda=X_\alpha\times_{Y_\alpha} Y_\lambda$ and let $f_\lambda\colon
X_\lambda\to Y_\lambda$ be the projection. Note that $X_\mu\to X_\lambda$
is schematically dominant since $f_\alpha$ is flat.
\end{proof}

This means that the restriction $\ev_\infty\colon \APP^\fl_\Lambda\to \Stk^\fl$ is
a fibered $2$-category where $\Stk^\fl$ is the non-full $2$-subcategory of
all algebraic stacks but with $1$-morphisms that are flat and of finite
presentation, and $\APP^\fl_\Lambda$ is the non-full $2$-subcategory with all
objects and $1$-morphisms $\{f_\lambda\}$ such that $f_\infty$ is flat and of
finite presentation. Equivalently, we have a $2$-functor
\[
\APP_\Lambda\colon (\Stk^\fl)^\op\to \Pos, \quad X\mapsto \APP_\Lambda(X)
\]
where $\Pos$ denotes the $2$-category of partially ordered sets.

\subsection{Descent}
We will now show that $\APP_\Lambda$ has faithfully flat descent.

\begin{theorem}\label{T:descent}
Let $f\colon X'\to X$ be a faithfully flat morphism of finite presentation. Then
$\APP_\Lambda$ is a stack with respect to $f$, that is,
\[
\xymatrix{%
\APP_\Lambda(X)\ar[r]
 & \APP_\Lambda(X')\ar@<.5ex>@{+->+}[r] \ar@<-.5ex>@{+->+}[r]
 & \APP_\Lambda(X'\times_X X')
}%
\]
is an equalizer of partially ordered sets. Equivalently, it is an equalizer of
sets such that $\APP_\Lambda(X) \to \APP_\Lambda(X')$ is
order-reflecting.
\end{theorem}
\begin{proof}
We first show that $\APP_\Lambda(X)\to \APP_\Lambda(X')$ is
order-reflecting. Let $\{ X_{1,\lambda}\}$ and $\{X_{2,\lambda}\}$ be
objects of $\APP_\Lambda(X)$ such that $f^*\{ X_{1,\lambda}\} \leq f^*\{
X_{2,\lambda}\}$.
Choose an algebraic stack $X_0$ of finite presentation over $\ZZ$
and an affine schematically dominant morphism $X\to X_0$, e.g., take
$X_0=X_{1,\alpha}$. After replacing $X_0$ with a finer approximation of $X$,
we can find a flat morphism $f_0\colon
X'_0\to X_0$ of finite presentation such that $X'=X'_0\times_{X_0} X$. Let $A$
and $A'=f_0^*A$ be the $\sO_{X_0}$- and $\sO_{X'_0}$-algebras corresponding to
$X$ and $X'$. By Corollary~\ref{C:approximations=colimits-of-subalgebras}
we can identify $\APP_\Lambda(X)^\op$ with the category $\APP_\Lambda(A)$ of
colimit diagrams $\{A_\lambda\}$ of
$\sO_{X_0}$-subalgebras of $A$ of almost shape $\Lambda$ and similarly for $X'$.
Since
$f^*\{A_{1,\lambda}\} \leq f^*\{A_{2,\lambda}\}$, there exists an index
$\gamma$ such that $f^*A_{1,\lambda}\subseteq f^*A_{2,\lambda}$ for all
$\lambda\geq \gamma$. By flat descent, it follows that $A_{1,\lambda}\subseteq
A_{2,\lambda}$ for all $\lambda\geq \gamma$. That is, 
$\{ X_{1,\lambda}\}\leq \{ X_{2,\lambda}\}$.

Since $\APP_\Lambda(X)\to \APP_\Lambda(X')$ is order-reflecting, it is also
injective. Let $X''=X'\times_X X'$ and let $X'''=X'\times_X X'\times_X X'$.
Let $\{ X'_\lambda \}\in \APP_\Lambda( X' )$ such that the two
pull-backs to $\APP_\Lambda( X'' )$ are equal. It remains to show that
$\{X'_\lambda\}$ comes from an object of $\APP_\Lambda(X)$.

The three pull-backs of $\{ X'_\lambda \}$
to $\APP_\Lambda( X''' )$ are also equal.
Choose representatives $\{ X''_\lambda \}$ and $\{ X'''_\lambda \}$ of the
pull-backs in $\APP_\Lambda$. Then we have
flat cartesian maps
\[
\xymatrix{%
 \{X'_\lambda\}
 & \{X''_\lambda\} \ar@<.5ex>@{+->+}[l] \ar@<-.5ex>@{+->+}[l]_{\pi_1}
 & \{X'''_\lambda\} \ar@<1ex>@{+->+}[l] \ar@{+->+}[l] \ar@<-1ex>@{+->+}[l]_{\pi_{12}}
}%
\]
in $\APP_\Lambda$ and
$\{X'''_\lambda\}=\{X''_\lambda\}\times_{\pi_1,\{X'_\lambda\},\pi_2}
\{X''_\lambda\}$. By Proposition~\ref{P:cartesian}, the diagonal $X'\to
X'\times_X X'$ induces a map $\Delta\colon \{X'_\lambda\}\to \{X''_\lambda\}$.
The maps $s=\pi_1$, $t=\pi_2$, $c=\pi_{13}$, $e=\Delta$ endows
$\{X''_\lambda\}\rightrightarrows \{X'_\lambda\}$ with the structure of a
groupoid in $\APP_\Lambda$. The axioms, which involves $2$-isomorphisms between 
various compositions and identities between $2$-isomorphisms and hold for
$X''\rightrightarrows X'$, are satisfied by
Proposition~\ref{P:fully-faithful-on-morphisms}.  Since the axioms involve a
finite number of morphisms and a finite number of compositions, we may find an
index $\alpha$ such that $X''_\lambda\rightrightarrows X'_\lambda$ becomes a
groupoid for every $\lambda\geq \alpha$. If $X_\lambda$ denotes the stack
quotient, then we obtain an element $\{X_\lambda\}_{\lambda\geq \alpha}$ of
$\APP_\Lambda(X)$ such that $f^*\{X_\lambda\}=\{X'_\lambda\}$.
\end{proof}

\end{section}


\begin{section}{Adjoints for pure morphisms}\label{S:pure-adjoint}
Let $f\colon X\to Y$ be a morphism of algebraic stacks and let $\sF$ be a
quasi-coherent $\sO_Y$-module. We obtain a functor
\[
\widetilde{f^*}\colon \Sub(\sF)\to \Sub(f^*\sF)
\]
taking a quasi-coherent $\sO_Y$-submodule $\sF_0\subseteq \sF$ to the image of
$f^*\sF_0\to f^*\sF$. When $f$ is flat, then $\widetilde{f^*}\sF_0=f^*\sF_0$.
When $f$ is quasi-compact and quasi-separated, then $\widetilde{f^*}\sF_0$ has
a \emph{right adjoint}
\[
\widetilde{f_*}\colon \Sub(f^*\sF)\to \Sub(\sF)
\]
taking $\sG_0\subseteq f^*\sF$ to $f_*\sG_0 \times_{f_*f^*\sF} \sF$.
Note that $\widetilde{f^*}$ always preserves submodules of finite type,
but in general $\widetilde{f_*}$ does not.
Since $\Sub(-)$ is a partially ordered set, $(\widetilde{f^*},\widetilde{f_*})$
is an example of a Galois connection and $\widetilde{f_*}\sG_0$ is the
largest submodule $\sF_0\subseteq \sF$ such that $\widetilde{f^*}\sF_0\subseteq
\sG_0$. Since $\widetilde{f^*}$ preserves unions, we also have that
$\widetilde{f_*}\sG_0$ is the union of all submodules $\sF_i\subseteq \sF$
such that
$\widetilde{f^*}\sF_i\subseteq \sG_0$.

If a \emph{left adjoint} $f_!$ to $\widetilde{f^*}$ exists, then $f_!\sG_0$ is
the intersection of all submodules $\sF_i\subseteq \sF$ such that
$\sG_0\subseteq \widetilde{f^*}\sF_i$. This intersection always makes sense,
but only defines a left adjoint if $\widetilde{f^*}f_!\sG_0$ contains
$\sG_0$. This is not always the case since $\widetilde{f^*}$ does not preserve
intersections in general, even for flat $f$
\cite[Ex.~6.1]{rydh_approximation-sheaves}.

Now assume that $f$ is flat, of finite presentation, and \emph{pure} in the
sense of Raynaud--Gruson \cite[D\'ef.\ 3.3.3]{raynaud-gruson},
\cite[Def.~4.8]{rydh_approximation-sheaves}. Then a left adjoint
\[
f_!\colon \Sub(f^*\sF)\to \Sub(\sF)
\]
exists~\cite[Thm.~6.3]{rydh_approximation-sheaves}. It also preserves submodules
of finite type and commutes with arbitrary base change $g\colon Y'\to Y$ in the
sense that $\widetilde{g^*}f_! = f'_! \widetilde{g'^*}$ where $f'$ and $g'$
denote the pull-backs of $f$ along $g$ and $g$ along
$f$~\cite[Thm.~6.3]{rydh_approximation-sheaves}.


This result immediately generalizes to subalgebras:

\begin{theorem}\label{T:left-adjoint-for-subalg}
Let $f\colon X\to Y$ be flat, of finite presentation, and \emph{pure}. Let
$\sA$ be a quasi-coherent $\sO_Y$-algebra. Then $f^*\colon \Sub(\sA)\to
\Sub(f^*\sA)$ has a left adjoint
\[
f^{\alg}_!\colon \Sub(f^*\sA)\to \Sub(\sA).
\]
Moreover, $f^{\alg}_!$ preserves subalgebras of finite type and commutes with
arbitrary base change.
\end{theorem}
\begin{proof}
Let $\sB_0\subseteq f^*\sA$ be a subalgebra. It is clear that $f^{\alg}_!\sB_0$
is the smallest subalgebra containing $f_!\sB_0\subseteq \sA$, that is,
$f^{\alg}_!\sB_0= \im\left( \Sym_{\sO_Y}(f_!\sB_0)\to \sA \right)$. Since
symmetric products and images commute with pull-backs, it is also clear that
$f^{\alg}_!$ commutes with arbitrary base change.

Now assume that $\sB_0$ is an $\sO_X$-algebra of finite type. To prove that
$f^{\alg}_!\sB_0$ is of finite type, we may work fppf-locally on $Y$ and assume
that $Y$ is affine. Then $X$ is
pseudo-noetherian~\cite[Prop.~4.8]{rydh_noetherian-approx}
so we may write $\sB_0$ as
the union of its $\sO_X$-submodules of finite type. In particular, there
is a submodule $\sG_0\subseteq \sB_0$ of finite type such that $\sB_0$ is the
smallest subalgebra containing $\sG_0$. Then $f^{\alg}_!\sB_0$ is the smallest
subalgebra containing $f_!\sG_0$, hence of finite type.
\end{proof}

\begin{remark}
The right adjoint $f_*$ for submodules is also a right adjoint for
subalgebras. In general, however, $f_*$ does not preserve algebras of
finite type.
\end{remark}

\begin{remark}\label{R:connected-fibration}
Let $f\colon X\to Y$ be a morphism of finite presentation between algebraic
stacks. If $f$ is smooth, or more generally flat with geometrically reduced
fibers, then there is a canonical factorization $X\to \pi_0(X/Y)\to Y$ where
$X\to \pi_0(X/Y)$ has geometrically connected fibers and $\pi_0(X/Y)\to Y$ is
\'etale and representable, but not necessarily
separated. See~\cite[6.8]{laumon} for $f$ smooth and representable
and~\cite[Thm.~2.5.2]{romagny_components-in-families} for the general case.
When $f$ is smooth, the morphism $X\to \pi_0(X/Y)$ is smooth with geometrically
connected fibers, hence pure~\cite[Ex.~3.3.4 (iii)]{raynaud-gruson}.
\end{remark}

\begin{theorem}\label{T:existence-of-adjoint}
Let $f\colon X\to Y$ be a faithfully flat morphism of finite presentation
between quasi-compact and quasi-separated algebraic stacks.
If $f$ is smooth with geometrically connected fibers,
then $f^*\colon \APP_\Lambda(Y)\to
\APP_\Lambda(X)$ admits a right adjoint $f_*$. Moreover, $f_*$ commutes with
base change along flat morphisms $g\colon Y'\to Y$ of finite presentation.
\end{theorem}

\begin{proof}
We first prove the theorem when $Y$ has an approximation, that is, when there
exists an affine schematically dominant morphism $h\colon Y\to Y_0$ with $Y_0$
of
finite presentation over $\Spec \ZZ$. We can arrange so that $f$ descends to a
smooth surjective morphism $f_0\colon X_0\to Y_0$ of finite
presentation~\cite[Prop.~B.3]{rydh_noetherian-approx}. We can also arrange so
that $f_0$ has geometrically connected fibers, e.g., using that we have a
factorization $X_0\to \pi_0(X_0/Y_0)\to Y_0$ which commutes with base change
(cf.\ Remark~\ref{R:connected-fibration}). Then $f_0$ is pure.

The preordered set $\APP_\Lambda(Y)^\op$ can be identified with the category of
colimit diagrams $\APP_\Lambda(A)$ of finitely generated $\sO_{Y_0}$-algebras
of almost shape $\Lambda$ and colimit $A=h_*\sO_Y$
(Corollary~\ref{C:approximations=colimits-of-subalgebras}). We have a similar
identification for $\APP_\Lambda(X)$ and $f^*$ takes an object
$\{A_\lambda\}\in \APP_\Lambda(A)$ to $\{f_0^*A_\lambda\}\in
\APP_\Lambda(f_0^*A)$. I claim that $f_!\colon \APP_\Lambda(f_0^*A)\to
\APP_\Lambda(A)$ given by $f_!(\{B_\lambda\})=\{(f_0)^\alg_!B_\lambda\}$ is a
left adjoint. This follows immediately from
Theorem~\ref{T:left-adjoint-for-subalg} except that we have to verify that
$\colim_\lambda (f_0)^\alg_!B_\Lambda=A$. Since $(f_0)^\alg_!$ is a left
adjoint, it commutes with colimits so $\colim_\lambda (f_0)^\alg_!B_\Lambda =
(f_0)^\alg_!f_0^*A$. Since $f_0$ is faithfully flat and
$f_0^*A\subseteq f_0^*(f_0)^\alg_!f_0^*A \subseteq f_0^*A$, it follows that
$(f_0)^\alg_!f_0^*A=A$.

Now, drop the assumption on $Y$. Let $g\colon Y'\to Y$ be a faithfully flat
morphism of finite presentation from an affine scheme $Y'$ and let $f'\colon
X'\to Y'$ be the base change of $f$ and $g'\colon X'\to X$ be the base change
of $g$. Then by the special case, $f'_*$ exists and commutes with flat base
change on $Y'$. We may now define $f_*$ by descending $f'_*g'^*$ along $g$
(Theorem~\ref{T:descent}) so that $g^*f_*=f'_*g'^*$ holds by definition. Since
$g^*$ and $g'^*$ are order-reflecting and $(f'^*,f'_*)$ is an adjunction, it
follows that $(f^*,f_*)$ is an adjunction.
\end{proof}

We can now prove the main theorem of the paper. Recall that it states that
an algebraic stack $X$ is quasi-compact and quasi-separated if and only if
it has an approximation, or equivalently, if and only if
$\APP_\Lambda(X)\neq \emptyset$ (Remark~\ref{R:approximation-vs-APP-nonempty}).

\begin{proof}[Proof of Theorem~\tref{T:MT}]
If $X$ has an approximation, then it is quasi-compact and quasi-separated by
definition. Conversely, if $X$ is quasi-compact and quasi-separated,
let $p\colon U\to X$ be a smooth presentation with $U$ affine. Then by
Remark~\ref{R:connected-fibration} we have a
factorization $p=gf\colon U\to X':=\pi_0(U/X)\to X$ where $f$ is smooth with
geometrically connected fibers and $g$ is \'etale and representable.
Since $U=\Spec A$ is affine, it can be approximated by $\Spec \ZZ$.
If $\Lambda$ is the partially ordered set of finitely
generated $\ZZ$-subalgebras of $A$, then we have a canonical element
$\{U_\lambda\}_{\lambda\in \Lambda}$ of $\APP_\Lambda(U)$. By
Theorem~\ref{T:existence-of-adjoint}
we have an element $f_*\{U_\lambda\}\in \APP_\Lambda(X')$. In particular,
$X'$ has an approximation. By definition, this means that
$X$ is of approximation type~\cite[Def.~2.9]{rydh_noetherian-approx} and hence
also has an approximation~\cite[Thm.~7.10]{rydh_noetherian-approx}.
\end{proof}

\end{section}


\begin{section}{Applications}\label{S:applications}

\subsection{Algebraicity of moduli spaces and stacks}
Algebraicity results for stacks with finite diagonals were obtained
in~\cite[Thms.~A~(i) \& B]{hall-rydh_gen-hilb-quot} without assuming
locally of finite presentation. With the new approximation result we
can drop the assumption on the diagonal.

\begin{theorem}
Let $f\colon X\to S$ be a separated morphism of algebraic stacks.
\begin{enumerate}
\item The Hilbert functor $\Hilb_{X/S}$ is an algebraic space, separated over
  $S$.
\item The stack $\Coh(X/S)$ is algebraic and has affine diagonal (relative to
  $S$).
\item If $\sF\in\QCoh(X)$, then the functor $\Quot(X/S,\sF)$ is representable
and separated over $S$.
\item The Hilbert stack $\HilbSt^{\qfin}_{X/S}$ is algebraic and has affine
  diagonal.
\end{enumerate}
\end{theorem}
\begin{proof}
By the main theorem, $f$ is locally of approximation type.
The algebraicity of $\Coh(X/S)$, $\Quot(X/S,\sF)$ and
$\Hilb_{X/S}=\Quot(X/S,\sO_X)$ thus follows
from~\cite[Thm.~4.4]{hall-rydh_gen-hilb-quot} and the algebraicity of
$\HilbSt^{\qfin}_{X/S}$ follows from~\cite[Thm.~2.2]{hall-rydh_gen-hilb-quot}.
\end{proof}

Similarly, we obtain the following strengthening of~\cite[Thms.~1.2 \&
  1.3]{hall-rydh_coherent-tannaka-duality} on the algebraicity of Hom-stacks
and Weil restrictions. Recall that a morphism $X\to Y$ has \emph{affine
  stabilizers} if the following equivalent conditions hold:
\begin{enumerate}
\item the diagonal $\Delta_{X/Y}$ has affine fibers;
\item the inertia stack $I_{X/Y}$ has affine fibers; and
\item for any field $k$ and point $x\colon \Spec k\to X$, the
automorphism group scheme $\Aut(x)$ is affine.
\end{enumerate}

\begin{theorem}
Let $S$ be an algebraic stack. Let $f\colon Z\to S$ be a proper and flat
morphism of finite presentation.
\begin{enumerate}
\item If $X\to S$ is a quasi-separated morphism with affine stabilizers, then
  the stack
  \[
  \underline{\Hom}_S(Z,X)\colon T\mapsto \Hom_S(Z\times_S T,X)
  \]
  is algebraic and $\underline{\Hom}_S(Z,X)\to S$ is quasi-separated with
  affine stabilizers. If $X\to S$ has affine/quasi-affine/separated diagonal,
  then so has $\underline{\Hom}_S(Z,X)\to S$.
\item If $X\to Z$ is a quasi-separated morphism such that $X\to Z\to S$ has
  affine stabilizers, then the Weil restriction
  \[
  f_*X=\weilr_{Z/S}(X)\colon T\mapsto \Hom_Z(Z\times_S T,X)
  \]
  is algebraic and $f_*X\to S$ is quasi-separated with affine stabilizers. If
  $X\to Z$ has affine/quasi-affine/separated diagonal, then so has $f_*X\to S$.
\end{enumerate}
\end{theorem}
\begin{proof}
By the main theorem, $X\to S$ is locally of approximation type. The result now
follows from~\cite[Cor.~9.2]{hall-rydh_coherent-tannaka-duality}.
\end{proof}

\subsection{Zariski's Main Theorem}
We obtain the following version of Zariski's Main Theorem, slightly
generalizing~\cite[Thm.~8.1]{rydh_approximation-sheaves}.

\begin{theorem}
Let $X\to S$ be a representable, quasi-finite and separated morphism.
If $S$ is quasi-compact and quasi-separated, then there exists a factorization
$X\to X'\to S$ where $X\to X'$ is an open immersion and $X'\to S$ is finite.
If in addition $X\to S$ is of finite presentation, we can arrange so that
$X'\to S$ also is of finite presentation.
\end{theorem}
\begin{proof}
By the main theorem, $S$ is pseudo-noetherian (Corollary~\tref{C:COMPLETE}).
The result is thus \cite[Thm.~8.6 (ii)]{rydh_noetherian-approx}.
\end{proof}


\subsection{Elimination of noetherian hypothesis in Chow's Lemma}
We can also remove the noetherian assumption of the main result
of~\cite{olsson_proper-coverings}:

\begin{theorem}\label{T:proper-covering}
Let $X$ be a quasi-compact separated algebraic stack. Then there exists
a proper surjective morphism $X'\to X$ where $X'$ is a separated scheme
which admits an ample line bundle.
\end{theorem}
\begin{proof}
Choose an approximation $X\to X_0$, that is, $X_0$ is of finite presentation
over $\Spec \ZZ$ and $X\to X_0$ is affine. We can also assume that $X_0$
is separated~\cite[Thm.~D]{rydh_noetherian-approx}.
By~\cite{olsson_proper-coverings}
there exists a proper surjective morphism $X'_0\to X_0$ where $X'_0$ is a
quasi-projective scheme. We can now take $X':=X'_0\times_{X_0} X$.
\end{proof}

\subsection{Approximation of proper morphisms}
\begin{corollary}
Let $S$ be a quasi-compact algebraic stack and let $X=\varprojlim_\lambda
X_\lambda$ be an inverse limit of finitely presented $S$-stacks such that
$X_\mu\to X_\lambda$ is a \emph{closed immersion} for every $\mu\geq
\lambda$. If $X\to S$ is proper, then so is $X_\lambda\to S$ for all
sufficiently large $\lambda$.
\end{corollary}
\begin{proof}
This follows as in the proof
of~\cite[Cor.~6.6]{rydh_noetherian-approx}, replacing the use of
\cite[Cor.~6.5 and Thm.~B]{rydh_noetherian-approx} with
\cite[Cor.~6.7]{rydh_noetherian-approx} and Theorem~\ref{T:proper-covering}.
\end{proof}

As a consequence, we can add the property \emph{proper} (not necessarily with
finite diagonal) to the list (PC) figuring
in~\cite[Thms.~C and~D]{rydh_noetherian-approx}.

\subsection{Applications to good moduli spaces}
Let $\stX$ be an algebraic stack. A \emph{good moduli space} for $\stX$
\cite{alper_good-mod-spaces}, \cite[1.7.3]{alper-hall-rydh_etale-local-stacks}
is an algebraic space $X$ together with a map $\pi\colon \stX\to X$
such that
\begin{enumerate}
\item $\pi$ is quasi-compact and quasi-separated,
\item $\pi_*\colon \QCoh(\stX)\to \QCoh(X)$ is exact, also after arbitrary base
  change $X'\to X$, and
\item the unit $\sO_X \to \pi_*\sO_\stX$ is an isomorphism.
\end{enumerate}
If a good moduli space exists, then it is unique
\cite[Thm.~3.12]{alper-hall-rydh_etale-local-stacks}.
The basic examples of good moduli spaces are
\begin{enumerate}
\item the GIT quotient $[\Spec A/G]\to \Spec A^G$ where $G$ is a
\emph{linearly reductive} group acting on an affine scheme $A$.
\item the GIT quotient $[X^{ss}(L)/G]\to X \gitq G$ where $G$ is a
\emph{linearly reductive} group acting on a polarized projective scheme
$(X,L)$.
\end{enumerate}
There is also a notion of \emph{adequate} moduli space which is equivalent to
good moduli space in characteristic zero but allows for arbitrary reductive
group actions in positive characteristic. If $\pi\colon \stX\to X$ is a good
(resp.\ adequate) moduli space, then $\pi$ is universally closed and every fiber
$\pi^{-1}(x)$ has a unique closed point $x_0$ which has linearly reductive
(resp.\ reductive) stabilizer.

A stack $\stX$ is \emph{fundamental} if it is of the form $[\Spec A / \GL_n]$
for some ring $A$ and $n\in \NN$. Then $\pi\colon \stX\to
\Spec(\stX,\sO_\stX)=\Spec A^{\GL_n}$ is an adequate moduli space. A stack
$\stX$ is \emph{linearly fundamental} if it is fundamental and $\pi$ is a good
moduli space. Equivalently, $\stX$ is fundamental (resp.\ linearly fundamental)
if it has an affine adequate (resp.\ good) moduli space, affine
stabilizers and the resolution property.

\subsubsection{\'Etale-local structure of good moduli spaces}
The following results generalize \cite[Thm.~6.4, Thm.~6.1, Cor.~6.11 and
Prop.~6.14]{alper-hall-rydh_etale-local-stacks} by removing the assumption
that $\stX$ is of finite presentation over some algebraic space.

\begin{theorem}\label{T:adequate-local-structure}
Let $\pi\colon \stX\to X$ be an adequate moduli space with $X$ quasi-separated.
Let $x\in |X|$ be a point. Assume that $\stX$ has affine stabilizers and
separated diagonal and that the unique closed point $x_0\in \pi^{-1}(x)$ has
linearly reductive stabilizer. Then there exists an \'etale neighborhood
$(X',x')\to (X,x)$, with $\kappa(x')=\kappa(x)$ such that $\stX'=\stX\times_X
X'$ is fundamental.
\end{theorem}
\begin{proof}
It is enough to prove that $\stX\times_X \Spec \sO_{X,x}^h$ is fundamental
since this property spreads
out~\cite[Lem.~2.15~(1)]{alper-hall-rydh_etale-local-stacks}. In particular,
we may assume that $x$ is closed.
 
The residual gerbe $\stG_{x_0}$ is linearly fundamental. We can thus apply the
non-noetherian local structure theorem~\cite[Thm.~5.1]{alper-hall-halpern-leistner-rydh_local-third}\footnote{This relies on the main theorem of this paper.},
to $\stW_0:=\stX_0:=\stG_{x_0}\inj \stX$. This gives an \'etale morphism
$f\colon \stW\to \stX$ such that $\stW$ is fundamental and $f|_{\stX_0}$ is an
isomorphism. Moreover, since $\stX$ has separated diagonal, we can arrange so
that $f$ is
representable~\cite[Prop.~5.7~(2)]{alper-hall-rydh_etale-local-stacks}. Let $W$
be the adequate moduli space of $\stW$. By Luna's fundamental
lemma~\cite[Thm.~3.14]{alper-hall-rydh_etale-local-stacks},
after replacing $W$ by an open neighborhood, it holds that
$\stW:=\stX\times_X W$ and that $W\to X$ is \'etale. The result follows with
$X':=W$.
\end{proof}

\begin{theorem}[Local structure of good moduli spaces]\label{T:local-structure-gms}
Let $\pi\colon \stX\to X$ be a good moduli space with $X$ quasi-compact and
quasi-separated. Assume that $\stX$ has affine stabilizers and separated
diagonal. Then there exists a Nisnevich covering $X'\to X$ such that
$\stX'=\stX\times_X X'$ is linearly fundamental. In particular, $\pi$ has
affine diagonal.
\end{theorem}
\begin{proof}
Since $\pi$ is a good moduli space, the unique closed point in every fiber has
linearly reductive stabilizer. By the previous theorem, we can thus find a
Nisnevich covering $X'\to X$ with $X'$ affine such that $\stX'$ is
fundamental. Since $\stX'\to X'$ is a good moduli space, $\stX'$ is linearly
fundamental.
\end{proof}

\begin{corollary}[Adequate with linearly reductive stabilizers is good]
Let $\pi\colon \stX\to X$ be an adequate moduli space with $X$ quasi-compact
and quasi-separated. Assume that $\stX$ has affine stabilizers and separated
diagonal. Then $\pi$ is a good moduli space if and only if every closed point
of $\stX$ has linearly reductive stabilizer.
\end{corollary}
\begin{proof}
It is enough to prove that $\pi$ is a good moduli space after replacing
$X$ with the henselization at any closed point of $X$. We can thus assume
that $X$ is local and henselian.
Theorem~\ref{T:adequate-local-structure} then tells us that
$\stX$ is fundamental and the result follows from
\cite[Cor.~6.10]{alper-hall-rydh_etale-local-stacks}.
\end{proof}

\begin{corollary}[Compact generation]
Let $\stX$ be a quasi-compact and quasi-separated algebraic stack with
affine stabilizers and separated diagonal. If $\stX$ admits a good moduli
space, then $\stX$ has the Thomason condition, that is:
\begin{enumerate}
\item $\Dqc(\stX)$ is compactly generated by a countable set of perfect
complexes; and
\item for every quasi-compact open substack $\stU\subseteq \stX$, there exists
  a compact perfect complex on $\stX$ with support $\stX\smallsetminus
  \stU$.
\end{enumerate}
\end{corollary}
\begin{proof}
Follows exactly as in \cite[Prop.~6.14]{alper-hall-rydh_etale-local-stacks}
using Theorem~\ref{T:local-structure-gms}.
\end{proof}

\subsubsection{Approximation of stacks with good moduli spaces}
The following two results generalize \cite[Thm.~7.3 and
  Cor.~7.4]{alper-hall-rydh_etale-local-stacks} by removing the assumption that
$\stX$ has the resolution property. By slight abuse of notation, we let $\stX$
``admits a good moduli space with affine diagonal'' mean that the morphism
$\pi\colon \stX\to X$ has affine diagonal, not that the good moduli space $X$
has affine diagonal.

\begin{theorem}\label{T:approx-gms}
Let $S$ be a quasi-compact algebraic space and let $\stX=\varprojlim_\lambda
\stX_\lambda$ be an inverse limit of quasi-compact and quasi-separated
morphisms $\{\stX_\lambda\to S\}$ of algebraic stacks with affine transition
maps. Suppose that $S$ satisfies (FC) or that $\stX$ satisfies (PC) or (N). If
$\stX$ admits a good moduli space with affine diagonal, then so does
$\stX_\lambda$ for all sufficiently large $\lambda$.
\end{theorem}
\begin{proof}
The question is local on $S$ so we can assume that $S$ is quasi-separated.
Before studying the system $\{\stX_\lambda\}$, we will show that there exists
an approximation of $\stX$ over $S$ with a good moduli space.

Let $\stX\to X$ denote the good moduli space. By
Theorem~\ref{T:local-structure-gms}, there exists an \'etale surjective morphism
$X'\to X$ such that $\stX':=\stX\times_X X'$ is linearly fundamental with
good moduli space $X'$. In particular, $X'$ is affine.

Write $\stX$ as an inverse limit of stacks $\stX_\mu\to X$ of finite
presentation with affine transition maps. For sufficiently large $\mu$, the
stack $\stX_\mu\times_X X'$ is linearly fundamental
\cite[Thm.~7.3]{alper-hall-rydh_etale-local-stacks}. Let $\stY:=\stX_\mu$ for
one such $\mu$ so that $\stY':=\stY\times_X X'$ is linearly fundamental.  Let
$Y:=\Spec_X p_*\sO_\stY$ where $p\colon \stY\to X$ is the structure map.  Let
$Y':=Y\times_X X'$. Then $\stY'\to Y'$ is a good moduli space with affine
diagonal. It follows that so is $\stY\to Y$.

Now write $Y$ as an inverse limit of stacks $Y_\alpha$ of finite presentation
over $S$ with affine transition maps. For sufficiently large $\alpha$ we have
an \'etale surjective morphism $Y'_\alpha\to Y_\alpha$ of finite presentation
and a finitely presented morphism $\stY_\alpha\to Y_\alpha$ that pull back to
$Y'\to Y$ and $\stY\to Y$ respectively. In particular, $\stY'_\alpha =
\stY_\alpha\times_{Y_\alpha} Y'_\alpha$, gives an inverse system with limit the
linearly fundamental stack $\stY'$. It follows that $\stY'_\alpha$ is linearly
fundamental for sufficiently large $\alpha$
\cite[Thm.~7.3]{alper-hall-rydh_etale-local-stacks}.  Arguing as before, we
conclude that $\stY_\alpha\to \Spec_{Y_\alpha} (p_\alpha)_*\sO_{\stY_\alpha}$
is a good moduli space with affine diagonal. Note that $\stX\to \stY\to
\stY_\alpha$ is affine.  We have thus obtained an approximation which admits a
good moduli space.

Since $\stY_\alpha\to S$ is of finite presentation, we obtain a map
$\stX_\lambda\to \stY_\alpha$ for all sufficiently large $\lambda$.  After
increasing $\lambda$, this map is
affine~\cite[Thm.~C]{rydh_noetherian-approx}. It follows that $\stX_\lambda$
has a good moduli space with affine diagonal.
\end{proof}

Theorem~\ref{T:approx-gms} says that ``having a good moduli space with affine
diagonal'' can be included in the list of properties (PA) figuring in
\cite[Thms.~C and~D]{rydh_noetherian-approx}, under the assumptions (FC), (PC)
or (N).

\begin{corollary}\label{C:approx-gms}
Let $\stX$ be a quasi-compact and quasi-separated algebraic stack that admits a
good moduli space with affine diagonal. Suppose that $\stX$ satisfies (FC),
(PC) or (N). Then there exists a stack $\stX_0$ that admits a good moduli space
with affine diagonal and an affine morphism $\stX\to \stX_0$ such that $\stX_0$
is of finite presentation over a localization of $\Spec \ZZ$.
\end{corollary}
\begin{proof}
If $\stX$ satisfies (FC), let $S$ be the semi-localization of $\Spec \ZZ$ in
the characteristics that appear in $\stX$. Otherwise, let $S=\Spec \ZZ$.
By the main theorem, $\stX$ can be written as an inverse limit of finitely
presented $S$-stacks with affine transition maps. The result now follows
from Theorem~\ref{T:approx-gms}.
\end{proof}

\subsubsection{Deformation of the resolution property}
The following is a variant of
\cite[Prop.~7.8]{alper-hall-rydh_etale-local-stacks} without excellency
assumptions. This implies that in
\cite[Setup.~7.6(c)]{alper-hall-rydh_etale-local-stacks} it is enough to
assume that $\stX_0$ has the resolution property.

\begin{proposition}
Let $\stX$ be an algebraic stack with affine diagonal and affine good moduli
space $X$. Let $\stX_0\inj \stX$ be a closed substack with good moduli space
$X_0$, which is a closed subscheme of $X$. Suppose that $(X,X_0)$ is an affine
henselian pair and that $\stX_0$ satisfies (FC), (PC) or (N). If $\stX_0$ has
the resolution property, then so does $\stX$.
\end{proposition}
\begin{proof}
Since $(X,X_0)$ is an henselian pair, $\stX$ also satisfies (FC), (PC) or (N)
\cite[Rmk.~7.1]{alper-hall-rydh_etale-local-stacks}. By the main theorem, we
can write $\stX_0\inj \stX$ as an inverse limit of finitely presented closed
immersions $\stX_\alpha\inj \stX$. For sufficiently large $\alpha$,
$\stX_\alpha$ also has the resolution
property~\cite[Lem.~2.15~(1)]{alper-hall-rydh_etale-local-stacks}. Note that
$(X,X_\alpha)$ also is a henselian pair.  After replacing $\stX_0$ with
$\stX_\alpha$ we can thus assume that $\stX_0\inj \stX$ is of finite
presentation.

By Corollary~\ref{C:approx-gms}, we have that $\stX=\varprojlim_\lambda
\stX_\lambda$ where $\stX_\lambda$ is essentially of finite presentation over
$\Spec \ZZ$ and admits a good moduli space. For sufficiently large $\lambda$ we
have a closed immersion $\stX_{\lambda,0}\inj \stX_{\lambda}$ such that $\stX_0
= \stX_{\lambda,0}\times_{\stX_\lambda} \stX$ and such that $\stX_{\lambda,0}$
has the resolution property.

Let $(X_\lambda,X_{\lambda,0})$ be the good moduli spaces of
$(\stX_\lambda,\stX_{\lambda,0})$. Then $X_\lambda$ is essentially of finite
type over $\Spec \ZZ$, hence excellent. We can thus apply \cite[Prop.~7.8 with
  Setup 7.6(b)]{alper-hall-rydh_etale-local-stacks} to deduce that
$\stX_\lambda\times_{X_\lambda} X_\lambda^h$ has the resolution property where
$X_\lambda^h$ is the henselization of $X_\lambda$ along $X_{\lambda,0}$. Since
$X\to X_\lambda$ factors through $X_\lambda^h$, we obtain an affine morphism
$\stX\to \stX_\lambda\times_{X_\lambda} X_\lambda^h$ and it follows that $\stX$
has the resolution property.
\end{proof}

\end{section}


\bibliography{noetherian-general}

\providecommand{\MR}{\relax\ifhmode\unskip\space\fi MR }
\providecommand{\MRhref}[2]{%
  \href{http://www.ams.org/mathscinet-getitem?mr=#1}{#2}
}
\providecommand{\href}[2]{#2}
\begin{thebibliography}{AHHLR22}

\bibitem[AHHLR22]{alper-hall-halpern-leistner-rydh_local-third}
Jarod Alper, Jack Hall, Daniel Halpern-Leistner, and David Rydh, \emph{Artin
  algebraization for pairs with applications to the local structure of stacks
  and {F}errand pushouts}, To appear in \emph{Forum Math.\ Sigma}, May 2022,
  \href{http://arXiv.org/abs/2205.08623}{\mbox{arXiv:2205.08623}}.

\bibitem[AHR19]{alper-hall-rydh_etale-local-stacks}
Jarod Alper, Jack Hall, and David Rydh, \emph{The {\'e}tale local structure of
  algebraic stacks}, Preprint, Dec 2019,
  \href{http://arXiv.org/abs/1912.06162}{\mbox{arXiv:1912.06162}}.

\bibitem[Alp13]{alper_good-mod-spaces}
Jarod Alper, \emph{Good moduli spaces for {A}rtin stacks}, Ann. Inst. Fourier
  (Grenoble) \textbf{63} (2013), no.~6, 2349--2402.

\bibitem[HR15]{hall-rydh_gen-hilb-quot}
Jack Hall and David Rydh, \emph{General {H}ilbert stacks and {Q}uot schemes},
  Michigan Math. J. \textbf{64} (2015), no.~2, 335--347.

\bibitem[HR19]{hall-rydh_coherent-tannaka-duality}
Jack Hall and David Rydh, \emph{Coherent {T}annaka duality and algebraicity of
  {H}om-stacks}, Algebra Number Theory \textbf{13} (2019), no.~7, 1633--1675.

\bibitem[LMB00]{laumon}
G{\'e}rard Laumon and Laurent Moret-Bailly, \emph{Champs alg\'ebriques},
  Springer-Verlag, Berlin, 2000.

\bibitem[Ols05]{olsson_proper-coverings}
Martin Olsson, \emph{On proper coverings of {A}rtin stacks}, Adv. Math.
  \textbf{198} (2005), no.~1, 93--106.

\bibitem[RG71]{raynaud-gruson}
Michel Raynaud and Laurent Gruson, \emph{Crit\`eres de platitude et de
  projectivit\'e. {T}echniques de ``platification'' d'un module}, Invent. Math.
  \textbf{13} (1971), no.~1, 1--89.

\bibitem[Rom11]{romagny_components-in-families}
Matthieu Romagny, \emph{Composantes connexes et irr\'eductibles en familles},
  Manuscripta Math. \textbf{136} (2011), no.~1-2, 1--32.

\bibitem[Ryd15]{rydh_noetherian-approx}
David Rydh, \emph{Noetherian approximation of algebraic spaces and stacks}, J.
  Algebra \textbf{422} (2015), 105--147.

\bibitem[Ryd16]{rydh_approximation-sheaves}
David Rydh, \emph{Approximation of sheaves on algebraic stacks}, Int. Math.
  Res. Not. IMRN (2016), no.~3, 717--737.

\end{thebibliography}
\bibliographystyle{dary}

\end{document}